\documentclass[12pt]{amsart}
\usepackage{amsmath,amsthm,amssymb,amscd,url,enumerate,mathtools}
\usepackage{graphicx}
\usepackage[margin=1in]{geometry}
\usepackage{afterpage}
\usepackage{tikz}
\usepackage[loadshadowlibrary]{todonotes}
\usepackage[all]{xy}
\usepackage[colorlinks=true,citecolor={blue},linkcolor = {purple},
pdfauthor={Hanson Smith}, pdftitle={Number Fields with Arbitrarily Large Minimal Index}, pdfsubject={Algebraic Number Theory},
            pdfkeywords={Radical extension, Pure extension, Monogenic, Index Form}]{hyperref} %I change the colors, the original is below. %Others are for metadata
\usepackage{tabularx}
\usepackage{subcaption} % for subfigures  
\usepackage{float}		% to get H positioning
\usepackage{booktabs}
\usepackage{subcaption}
\usepackage{caption}

\usepackage{enumitem} % For custom enumerate labels\
\usepackage{refcount} % for getting item numbers
\usepackage{etoolbox} % for patching
\usepackage{multirow}

\usepackage{multicol}
\usepackage{relsize} %for sizing math stuff. Use \mathsmaller and \mathlarger

\usepackage{dutchcal}

\usepackage[square]{natbib}
\setcitestyle{numbers}

\newcommand{\TITLE}{Number Fields with Arbitrarily Large Minimal Index}
\newcommand{\TITLERUNNING}{Number Fields with Arbitrarily Large Minimal Index}

\theoremstyle{plain}
\newtheorem{theorem}{Theorem}

\newtheorem{proposition}[theorem]{Proposition}
\newtheorem{lemma}[theorem]{Lemma}

\theoremstyle{definition}
\theoremstyle{remark}
\newtheorem{remark}[theorem]{Remark}

\newtheorem{example}[theorem]{Example}

\numberwithin{theorem}{section}

  {\end{list}}

  {\end{list}}
\newcommand{\tightoverset}[2]{%
  \mathop{#2}\limits^{\vbox to -.5ex{\kern-1.05ex\hbox{$#1$}\vss}}}

\numberwithin{equation}{section} %Number equations withing sections.

\def\Bcal{{\mathcal B}}

\def\Ocal{{\mathcal O}}

\newcommand{\FF}{\mathbb{F}}

\newcommand{\QQ}{\mathbb{Q}}

\newcommand{\ZZ}{\mathbb{Z}}

\newcommand{\Ind}{\operatorname{Ind}}

\title[\TITLERUNNING]{\TITLE}
\author[Hanson Smith]{Hanson Smith}
\address{Department of Mathematics, California State University San Marcos,
333 S. Twin Oaks Valley Rd.
San Marcos, CA 92096
USA}
\email{hsmith@csusm.edu}

\keywords{Monogenic, Index Form, Radical extension, Pure extension, Minimal Index} %Prime splitting, Prime ideal factorization} %, Monogenic} Root Extension
\subjclass[2020]{11R04}
\begin{document}

\sloppy %This helps with words running into the margins!

%%%%% To ease editing, for IMPAN journals add:

\baselineskip=17pt

%%%%%%%%%%%%%%%%%%%%%%%%%%%%%%%%%%%%%%%%%%%%%%%%%%%%%%%%%%%%%%%%%%%%%%%%%%%%%%%%%%%%%%%%%%%%%%%%%%%%%%%%%%%%%%%

\begin{abstract}
For a number field $K/\mathbb{Q}$, the minimal index is the least positive integer $m$ for which there exists a monogenic order with index $m$ in the maximal order. %$\mathbb{Z}[\alpha]$ with index $\big[\mathcal{O}_K:\mathbb{Z}[\alpha]\big]=m$, where $\mathcal{O}_K$ is the maximal order. 
For any $n>2$ and $N>1$, we construct infinitely many number fields of degree $n$ with minimal index greater than $N$. 
%Moreover, these number fields are exceptional in that they have no common index divisors.
\end{abstract}

%%%%%%%%%%%%%%%%%%%%%%%%%%%%%%%%%%%%%%%%%%%%%%%%%%%%%%%%%%%%%%%%%%%%%%%%%%%%%%%%%%%%%%%%%%%%%%%%%%%%%%%%%%%%%%

\maketitle

%%%%%%%%%%%%%%%%%%%%%%%%%%%%%%%%%%%%%%%%%%%%%%%%%%%%%%%%%%%%%%%%%%%%%%%%%%%%%%%%%%%%%%%%%%%%%%%%%%%%%%%%%%%%%%

\section{Introduction and Main Result}

%%%%%%%%%%%%%%%%%%%%%%%%%%%%%%%%%%%%%%%%%%%%%%%%%%%%%%%%%%%%%%%%%%%%%%%%%%%%%%%%%%%%%%%%%%%%%%%%%%%%%%%%%%%%%%

If $K$ is a number field with ring of integers $\Ocal_K$, then it is natural to be curious about the ways that one may describe $\Ocal_K$. When $\Ocal_K$ is generated by adjoining a single algebraic integer to $\ZZ$, then we say $K$ is \textit{monogenic}. As Dedekind \cite{Dedekind} showed with the polynomial $x^3-x^2-2x-8$, not every number field is monogenic. In this case, the integral prime 2 is a \textit{common index divisor}: 2 divides the index $\big[\Ocal_K:\ZZ[\alpha]\big]$ of any monogenic order in the maximal order.

Pleasants \cite{Pleasants} building on the work of Hensel \cite{Hensel1894} and Dedekind \cite{Dedekind} completely classifies common index divisors based on splitting information. Conversely, Engstrom \cite{Engstrom} verifies a conjecture of Ore by showing that the largest power of a common index divisor dividing the index of any monogenic order is not determined by splitting information. The recent book by Gouv\^ea and Webster \cite{CIDDBook} provides a nice historical account of the impact of common index divisors on the development of algebraic number theory in addition to translations of \cite{Dedekind} and \cite{Hensel1894}.

The \textit{field index} of a number field $K$ is defined to be 
\[\Ind_{\gcd}(K)\coloneqq \gcd_{\alpha\in \Ocal_K, \ K=\QQ(\alpha)}\Big(\big[\Ocal_K:\ZZ[\alpha]\big]\Big).\]
A common index divisor is a prime divisor of the field index. The number fields with field index 1 are exactly the number fields with no common index divisors, and by the results outlined above this is completely understood via the splitting of integral primes. However, there are non-monogenic number fields with field index equal to 1. Adjoining $\sqrt[3]{5\cdot 49}$ to $\QQ$ yields a well-known example. To better understand this phenomenon, we define the \textit{minimal index}:
\[\Ind_{\min}(K)\coloneqq \min_{\alpha\in \Ocal_K, \ K=\QQ(\alpha)}\Big(\big[\Ocal_K:\ZZ[\alpha]\big]\Big).\] 
For example, one can show that $\QQ\big(\sqrt[3]{5\cdot 49}\big)$ has field index 1 but minimal index 2. 

Since every quadratic extension of $\QQ$ is monogenic, the minimal index only becomes interesting for extensions of degree $n>2$. An early investigation by Hall \cite{Hall} constructs radical/pure\footnote{Fields generated by an $n^\text{th}$ root.} cubic fields with arbitrarily large minimal index. Much of the subsequent literature regarding minimal indices has focused on number fields of a fixed degree. For example, Spearman, Yang, and Yoo \cite{BYY} show that any cube-free positive integer can be realized as the minimal index of a radical cubic field. For number fields of a general degree, Thunder and Wolfskill \cite{ThunderWolfskill} give upper bounds for $\Ind_{\min}(K)$ in terms of the degree and discriminant of $K$, while Kim and Wolske \cite{KimWolske} find families of number fields containing a quadratic subfield where the minimal index is large relative to the discriminant.

The purpose of this note is to give a constructive proof of the following:

\begin{theorem}\label{Thm. Main}
    Let $n>2$ and $N>1$. Then there exist infinitely many number fields of degree $n$ with minimal index greater than $N$.
\end{theorem}

Theorem \ref{Thm. Main} is proved by combining two ingredients. The first ingredient is the construction in Proposition \ref{Prop. Conditions} of a number field $\QQ\big(\sqrt[n]{\ell p q^{n-1}}\big)$ with large enough minimal index. The second is Lemma \ref{Lem. Infinite} which establishes the infinitude of integers $\ell$, $p$, and $q$ satisfying the requisite conditions for Proposition \ref{Prop. Conditions} and giving distinct number fields. 

%%%%%%%%%%%%%%%%%%%%%%%%%%%%%%%%%%%%%%%%%%%%%%%%%%%%%%%%%%%%%%%%%%%%%%%%%%%%%%%%%%%%%%%%%%%%%%%%%%%%%%%%%%%%%%%%%%%%%%%%%%%%%%

\section{A Construction of Number Fields with Arbitrarily Large Minimal Index}

 To start, we establish some notation. If $k>1$, write $\ZZ/k\ZZ^\times$ for the multiplicative group and $\left(\ZZ/k\ZZ^\times\right)^d$ for the subgroup of $d^\text{th}$ powers. For a fixed $n>2$, define $d=d(n)$ to be $n$ if $n$ is odd and $\frac{n}{2}$ if $n$ is even. This ensures that $d$ divides $\frac{n(n-1)}{2}$ but not $\frac{(n-1)(n-2)}{2}$. We begin with our construction.

\begin{proposition}\label{Prop. Conditions}
    Fix $n>2$ and $N>1$, and write $d=d(n)$ as above. %Define $d = \begin{cases}        n \text{ if } n \text{ is odd }\\        \frac{n}{2} \text{ if } n \text{ is even}.    \end{cases}$ 
    Let $p$ and $q$ be primes and let $\ell$ be a squarefree integer coprime to $p$ and $q$ such that the following conditions are satisfied:
    \begin{enumerate}[label=(\arabic*), ref=(\arabic*)]
        \item $p>N$ and $p\equiv 1\bmod n$. \label{(1)}
        \item $q\equiv 1\bmod pn$, $\overline{-1}\in \left(\ZZ/q\ZZ^\times\right)^d$, and if $-p<a<p$ with $\overline{a}\in\left(\ZZ/p\ZZ^\times\right)^d$, then $\overline{a}\in\left(\ZZ/q\ZZ^\times\right)^d$. \label{(2)}
        \item $(\ell p)^{\frac{(n-1)(n-2)}{2}}$ is not a $d^\text{th}$ power modulo $q$.\label{(3)}
        \item $\big(\ell pq^{n-1}\big)^r\not\equiv  \ell p q^{n-1} \bmod r^2$ for any prime $r$ dividing $n$.\label{(4)}
    \end{enumerate}
    Then $\QQ\big(\sqrt[n]{\ell p q^{n-1}}\big)$ has minimal index greater than $N$.
\end{proposition}

%%%%%%%%%%%%%%%%%%%%%%%%%%%
%\htodo{\ref{(1)} ensures $p$ is big and has a proper subgroup of $d$th powers. $q\equiv 1\bmod p$ ensures $F$ will be a $d$th power modulo $p$ and that $q>p$. $q\equiv 1\bmod n$ ensures that $q$ has a proper subgroup of $d$th powers. Making the small residues of $d$th powers modulo $p$ and $-1$ $d$th powers modulo $q$ ensures that we will need a big index to satisfy the conditions. \ref{(3)} ensures that $F$ is not a $d$th power modulo $q$. \ref{(4)} allows us to have a nice integral basis.}

\begin{proof}
    First, we will compute the index form for $\QQ\big(\sqrt[n]{\ell p q^{n-1}}\big)$. Condition \ref{(4)} ensures that $x^n - \ell pq^{n-1}$ is maximal at each prime dividing $n$. (See \cite{g17} or \cite{JhorarKhanduja}.) Further, $x^n - \ell pq^{n-1}$ is Eisenstein and hence maximal at $p$ and the prime divisors of $\ell$. The only prime divisors of the discriminant of $x^n - \ell pq^{n-1}$ are the prime divisors of $n$ and $\ell$ and the primes $p$ and $q$. Thus $q$ is the only prime for which $x^n - \ell pq^{n-1}$ is not maximal. Noting $\QQ\big(\sqrt[n]{\ell p q^{n-1}}\big)=\QQ\big(\sqrt[n]{(\ell p)^{n-1}q}\big)$, we see $x^n-(\ell p)^{n-1} q$ is $q$-Eisenstein and hence maximal at $q$. Thus, 
\[\Bcal = \left\{1, \sqrt[n]{\ell p q^{n-1}}=(\ell p)^{\frac1n}q^{\frac{n-1}{n}}, (\ell p)^{\frac2n}q^{\frac{n-2}{n}}, \dots, (\ell p)^{\frac{n-1}{n}} q^{\frac{1}{n}}  \right\}\]
is an integral basis for $\QQ\big(\sqrt[n]{\ell p q^{n-1}}\big)$. 

If we have a monogenic order $\ZZ[\alpha]$ with monogenerator \[\alpha=X_1 (\ell p)^{\frac1n}q^{\frac{n-1}{n}} + X_2 (\ell p)^{\frac2n}q^{\frac{n-2}{n}}+\cdots + X_{n-1}(\ell p)^{\frac{n-1}{n}}q^{\frac{1}{n}},\] (we lose no generality assuming $X_0=0$), then the following change of basis matrix has determinant $\pm \big[\Ocal_K:\ZZ[\alpha]\big]$:

\begin{equation}\label{Eq: IndMatrix}
\mathcal{M}_{\operatorname{Ind}}\coloneqq\begin{bmatrix}  
    1     		 & 0            & 0       & \dots      & 0       & 0 \\
    0      	  	 & X_1          & X_2     & \dots      & X_{n-2} & X_{n-1} \\
    A_{3,1}	 	 & A_{3,2}	    & A_{3,3} & \dots     & A_{3,n-1}  & A_{3,n} \\
    A_{4,1}      & A_{4,2}      & A_{4,3} & \dots    & A_{4,n-1}  & A_{4,n} \\  
    \vdots       & \vdots       & \vdots  & \ddots    & \vdots  & \vdots \\
    A_{n-1,1}      & A_{n-1,2}      & A_{n-1,3} & \dots    & A_{n-1,n-1}  & A_{n-1,n} \\  
    A_{n,1}      & A_{n,2}      & A_{n,3} & \dots    & A_{n,n-1}  & A_{n,n} \\  
\end{bmatrix}.
\end{equation}

Here the entry $A_{i,j}$ is the coefficient of the $j^{\text{th}}$ element of $\Bcal$ when we compute 
\begin{equation}\label{Eq. Aij}
\left(X_1 (\ell p)^{\frac1n}q^{\frac{n-1}{n}} + X_2 (\ell p)^{\frac2n}q^{\frac{n-2}{n}}+\cdots + X_{n-1}(\ell p)^{\frac{n-1}{n}}q^{\frac{1}{n}}\right)^{i-1}.
\end{equation}

Writing $K=\QQ\big((\ell p)^\frac{1}{n}q^\frac{n-1}{n}\big)$, the index form of $K$ corresponding to the basis $\Bcal$ is the determinant of $\mathcal{M}_{\Ind}$. Indeed, $\big[\Ocal_K:\ZZ[\alpha]\big]=\big|\det\mathcal{M}_{\operatorname{Ind}}\big|$. %the matrix \eqref{Eq: IndMatrix}. 
%I.e., 
%\[\text{Index Form}\left(\QQ\left((\ell p)^\frac{1}{n}q^\frac{n-1}{n}\right)\right)=\Big|\det\mathcal{M}_{\operatorname{Ind}}\Big|.\]
We will avoid the absolute value and employ the notation 
\[F=F(X_1,\dots, X_{n-1})\coloneqq \det \mathcal{M}_{\operatorname{Ind}}.\] 
We aim to compute $F$ modulo $p$ and modulo $q$. We will only detail this computation for $p$, since after reordering the basis elements the same computation works for $q$ up to a sign, and $\overline{-1}\in \left(\ZZ/q\ZZ^\times\right)^d$. Expanding across the first row of \eqref{Eq: IndMatrix}, we see it suffices to compute the determinant of 
\[M\coloneqq\begin{bmatrix}  
X_1           & X_2     & X_3     & \dots      & X_{n-2} & X_{n-1} \\
A_{3,2}	      & A_{3,3} & A_{3,4} & \dots     & A_{3,n-1}  & A_{3,n} \\
A_{4,2}       & A_{4,3} & A_{4,4} & \dots    & A_{4,n-1}  & A_{4,n} \\  
\vdots        & \vdots  & \vdots    & \ddots &\vdots  & \vdots \\
A_{n-1,2}       & A_{n-1,3} & A_{n-1,4} &\dots    & A_{n-1,n-1}  & A_{n-1,n} \\  
A_{n,2}       & A_{n,3} & A_{n,4} & \dots    & A_{n,n-1}  & A_{n,n} \\  
\end{bmatrix}\]

Noting that products %of basis elements 
with $p$-adic valuation 1 or greater result in an entry that is divisible by $p$, we see that $M$ is upper triangular modulo $p$. In other words, if $i>j$, the entry $A_{i,j}$ will be divisible by $p$. Indeed, in this case, \eqref{Eq. Aij} shows each summand making up $A_{i,j}$ is coming from a product of basis elements resulting in a $p$-adic valuation of 1 or greater. Thus, the determinant of $M$ modulo $p$ is $X_1A_{3,3}\cdots A_{n,n}$. Again, considering \eqref{Eq. Aij} for $A_{i,i}$ and eliminating summands that are divisible by $p$ shows that we have
\[\det M\equiv\det\begin{bmatrix}    
 X_1           & X_2     & X_3 & \dots      & X_{n-2} & X_{n-1} \\
0	      & qX_1^2 &  A_{3,4} & \dots     & A_{3,n-1}  & A_{3,n} \\
0       & 0 &  q^2X_1^3 & \cdots & A_{4,n-1} & A_{4,n} \\  
\vdots        & \vdots & \vdots & \ddots    & \vdots  & \vdots \\
0       & 0 & 0 &\dots    & q^{n-3}X_1^{n-2}  & A_{n-1,n} \\  
0       & 0 & 0 & \dots    & 0  & q^{n-2}X_1^{n-1} \\  
\end{bmatrix}
\equiv q^{(n-1)(n-2)/2}X_1^{n(n-1)/2}\bmod p.\]
Similarly, 
\[F\equiv \pm(\ell p)^{(n-1)(n-2)/2}X_{n-1}^{n(n-1)/2}\bmod q.\]

Condition \ref{(2)} shows $q\equiv 1\bmod p$, so $F$ is a $d^\text{th}$ power modulo $p$. More precisely, any choice of $(z_1,\dots, z_{n-1})\in \ZZ^{n-1}$ such that $F(z_1,\dots, z_{n-1})\not\equiv 0\bmod p$ results in $F(z_1,\dots, z_{n-1})~\in~ \left(\ZZ/p\ZZ^\times\right)^d$. Condition~\ref{(2)} shows that $\pm X_{n-1}^{n(n-1)/2}$ will always be a $d^\text{th}$ power modulo $q$. However, Condition~\ref{(3)} shows that $F(z_1,\dots, z_{n-1})\notin \left(\ZZ/q\ZZ^\times\right)^d$. 

Suppose $(y_1,\dots y_{n-1})\in \ZZ^{n-1}$ are such that $|F(y_1,\dots y_{n-1})|\neq 0$ is minimal, and let $F_{\min}$ denote $F(y_1,\dots y_{n-1})$. I.e., $\Ind_{\min}(K)=|F_{\min}|$. Note that $F(y_1,\dots y_{n-1})=0$ implies $\ZZ[\alpha]$ is not an order in $\Ocal_K$. %the minimal index of $\QQ\big(sqrt[n]{\ell p q^{n-1}}\big)$ before taking. 
Condition \ref{(2)} implies $q>p$, so if $-p<F_{\min}<p$, then $F_{\min}\in\big\{ -p< a< p, \ \overline{a}\in \left(\ZZ/p\ZZ^\times\right)^d\big\}$. 
Thus Condition \ref{(2)} shows $F_{\min}\in \left(\ZZ/q\ZZ^\times\right)^d$. However, we have just seen that $F\notin \left(\ZZ/q\ZZ^\times\right)^d$. Thus $F_{\min}\geq p>N$ or $F_{\min}\leq -p<-N$, so $\Ind_{\min}(K)>N$. 
\end{proof}

All that remains is to show that the conditions in Proposition \ref{Prop. Conditions} are not too stringent.

\begin{lemma}\label{Lem. Infinite}
    There exist infinitely many triples $\ell$, $p$, $q$ satisfying the conditions of Proposition \ref{Prop. Conditions} and yielding distinct number fields.
\end{lemma}
    
\begin{proof}
    Dirichlet gives a density of choices of $p$ with $p>N$ and $p\equiv 1\bmod n$. Thus, we need only show that for any such $p$, there exist $\ell$ and $q$ satisfying Conditions \ref{(2)}, \ref{(3)}, and \ref{(4)}. If we have a prime $q$ satisfying \ref{(2)}, then Condition \ref{(3)} amounts to choosing $\ell$ modulo $q$ and Condition \ref{(4)} amounts to choosing $\ell$ modulo $n^2$. The Chinese (Sunzi's) remainder theorem and Dirichlet give us a density of choices of $\ell$ prime (and hence squarefree). For fixed $p$ and $q$, ranging over distinct choices of a prime $\ell$ yields number fields with distinct ramified primes, hence distinct number fields. Thus the proof of the lemma reduces to establishing Condition~\ref{(2)}. Here we will invoke Chebotarev, though we only require an infinitude of totally split primes in a Galois extension. 

    We consider the finite Galois extension 
    \[L=\QQ\left(\zeta_{pn},\sqrt[d]{-1},\sqrt[d]{a}: -p<a<p \text{ and } \overline{a}\in \left(\ZZ/p\ZZ^\times\right)^d\right).\] 
    Chebotarev's density theorem shows that the density of primes that split completely in $L$ is $\frac{1}{[L:\QQ]}$. If $q>p$ is one such prime, then $q$ splits completely in any subfield of $L$. Hence, $q$ splits completely in $\QQ(\zeta_{pn})$, so $q\equiv 1\bmod pn$. Likewise, $q$ splits completely in $\QQ(\sqrt[d]{a})$ for each $a$ with $-p<a<p$ and $\overline{a}\in \left(\ZZ/p\ZZ^\times\right)^d$ and for $a=-1$. %If $a$ is a $d^\text{th}$ power in $\ZZ$, then we have nothing to show, so we suppose this is not the case. Using Dedekind--Kummer factorization and the fact that $q$ does not divide the discriminant of $x^d-a$, 
    Complete splitting in $\QQ(\sqrt[d]{a})$ implies that $\sqrt[d]{a}$ reduces to a root of $x^d-a$ in $\FF_q$. Thus $\overline{a}\in\left(\ZZ/q\ZZ^\times\right)^d$. Any sufficiently large $q$ that splits completely in $L$ satisfies Condition~\ref{(2)}. Thus we have proved the lemma.
\end{proof}

\begin{remark}
    Proposition \ref{Prop. Conditions} was inspired by the work of Pleasants \cite{Pleasants} on exceptional\footnote{An \textit{exceptional} number field has no common index divisors but is not monogenic.} number fields and builds off the construction in Section 6 of \cite{ScofieldSmith}. In fact, the number fields constructed in Proposition \ref{Prop. Conditions} have no common index divisors. Indeed $x^n-\ell p q^{n-1}$ is maximal at all primes aside from $q$, so Dedekind--Kummer factorization shows the splitting of any integral prime $r\neq q$ mirrors the splitting of $x^n-\ell p q^{n-1}$ in $\FF_r[x]$. Thus $r$ is not a common index divisor. On the other hand, $x^n-(\ell p)^{n-1}q$ is maximal at $q$, so the same reasoning shows $q$ is not a common index divisor. 
    %can be generalized, and the conditions can be weakened in various manners. Most obviously, we can simply take $\ell$ to be square-free and coprime to $p$ and $q$. We have presented it in the manner above for simplicity and conciseness.
\end{remark}

We conclude with an explicit example.

\begin{example}
    Suppose we want to use Proposition \ref{Prop. Conditions} to construct a number field of degree $5$ with minimal index greater than 10. We fix $n=5$ and $N=10$. Conveniently, we can take $p=11$. The only $5^\text{th}$ powers modulo 11 between $-11$ and 11 are $-10, -1, 1,$ and 10. Thus $L=\QQ\big(\zeta_{55},\sqrt[5]{10}\big)$. One finds that $q=3191$ splits completely in $L$, and we check that $-10, -1, 1,$ and 10 are all $5^\text{th}$ powers modulo 3191. Conveniently, $11\cdot 3191^4\not\equiv \big(11\cdot 3191^4\big)^5\bmod 25$, so we can simply take $\ell=1$.

    Our desired number field is $\QQ\big(\sqrt[5]{11\cdot 3191^4}\big)$. The form $\det \mathcal{M}_{\Ind}$ takes 15 lines as an output in SageMath; however, modulo $11$ it is simply $X_1^{10}$. Modulo 3191, it is $556X_4^{10}$. We check that $\overline{556}\notin \left(\ZZ/3191\ZZ^\times\right)^5$. Therefore, the minimal index of $\QQ\big(\sqrt[5]{11\cdot 3191^4}\big)$ is greater than 10.
\end{example}

%%%%%%%%%%%%%%%%%%%%%%%%%%%%%%%%%%%%%%%%%%%%%%%%%%
%%%%%%%%%%%%%%%%%%%%%%%%%%%%%%%%%%%%%%%%%%%%%%%%%%%%%%%%%%%%%%%%%%%%%%%%%%%%%%%%%%%%%%%%%%%%%%%%%%%%%%%%%%%%
%%%%%%%%%%%%%%%%%%%%%%%%%%%%%%%%%%%%%%%%%%%%%%%%%%%%%%%%%%%%%%%%%%%%%%%%%%%%%%%%%%%%%%%%%%%%%%%%%%%%%%%%%%%%%

\bibliography{Bibliography}

\begin{thebibliography}{SYY16}

\bibitem[Ded78]{Dedekind}
Richard Dedekind.
\newblock \"{U}ber den {Z}usammenhang zwischen der {T}heorie der {I}deale und der {T}heorie der h\"{o}heren {K}ongruenzen.
\newblock {\em Abhandlungen der K\"oniglichen Gesellschaft der Wissenschaften zu G\"{o}ttingen}, 23(3):3--38, 1878.

\bibitem[Eng30]{Engstrom}
H.~T. Engstrom.
\newblock On the common index divisors of an algebraic field.
\newblock {\em Transactions of the American Mathematical Society}, 32(2):223--237, 1930.

\bibitem[Gas17]{g17}
T.~Alden Gassert.
\newblock A note on the monogeneity of power maps.
\newblock {\em Albanian J. Math.}, 11(1):3--12, 2017.

\bibitem[GW25]{CIDDBook}
Fernando~Q. Gouv\^{e}a and Jonathan Webster.
\newblock {\em Common inessential discriminant divisors---scenes from the early history of algebraic number theory}, volume~47 of {\em History of Mathematics}.
\newblock American Mathematical Society, Providence, RI, 2025.

\bibitem[Hal37]{Hall}
Marshall Hall.
\newblock Indices in cubic fields.
\newblock {\em Bull. Amer. Math. Soc.}, 43(2):104--108, 1937.

\bibitem[Hen94]{Hensel1894}
K.~Hensel.
\newblock Arithmetische {U}ntersuchungen \"{u}ber die gemeinsamen ausserwesentlichen {D}iscriminantentheiler einer {G}attung.
\newblock {\em Journal f\"{u}r die Reine und Angewandte Mathematik. [Crelle's Journal]}, 113:128--160, 1894.

\bibitem[JK17]{JhorarKhanduja}
Bablesh Jhorar and Sudesh~K. Khanduja.
\newblock On the index theorem of {O}re.
\newblock {\em Manuscripta Mathematica}, 153(1-2):299--313, 2017.

\bibitem[KW18]{KimWolske}
Henry~H. Kim and Zack Wolske.
\newblock Number fields with large minimal index containing quadratic subfields.
\newblock {\em Int. J. Number Theory}, 14(9):2333--2342, 2018.

\bibitem[Ple74]{Pleasants}
P.~A.~B. Pleasants.
\newblock The number of generators of the integers of a number field.
\newblock {\em Mathematika}, 21(2):160–167, 1974.

\bibitem[SS26]{ScofieldSmith}
Dylan Scofield and Hanson Smith.
\newblock Prime splitting and common $n$-index divisors in radical extensions: Part $p=2$, 2026.
\newblock https://arxiv.org/abs/2512.23677.

\bibitem[SYY16]{BYY}
Blair~K. Spearman, Qiduan Yang, and Jeewon Yoo.
\newblock Minimal indices of pure cubic fields.
\newblock {\em Arch. Math. (Basel)}, 106(1):35--40, 2016.

\bibitem[TW96]{ThunderWolfskill}
Jeffrey~Lin Thunder and John Wolfskill.
\newblock Algebraic integers of small discriminant.
\newblock {\em Acta Arith.}, 75(4):375--382, 1996.

\end{thebibliography}
\bibliographystyle{alpha}
%abbrv

%%%%%%%%%%%%%%%%%%%%%%%%%%%%%%%%%%%%%%%%%%%%%%%%%%%%%%%%%%%%%%%%%%%%%%%%%%%%%%%%%%%%%%%%%%%%%%%%%%%%%%%%%%%%%

\end{document}